\title{\bf Nondeterministic graph property testing}
\author{L\'aszl\'o Lov\'asz\\
Katalin Vesztergombi\\
Institute of Mathematics, E\"otv\"os Lor\'and University\\
Budapest, Hungary}

\documentclass{article}
\usepackage{amssymb,amsmath,graphicx,bbm}
\usepackage{theorem}
\usepackage[hypertex]{hyperref}

\long\def\ignore#1{}

\nonstopmode

\newtheorem{theorem}{Theorem}[section]
\newtheorem{prop}[theorem]{Proposition}
\newtheorem{lemma}[theorem]{Lemma}

\newtheorem{corollary}[theorem]{Corollary}

\theorembodyfont{\rmfamily}

\newenvironment{proof}{\medskip\noindent{\bf Proof. }}{\hfill$\square$\medskip}
\newenvironment{proof*}[1]{\medskip\noindent{\bf Proof of #1.}}{\hfill$\square$\medskip}

\long\def\killtext#1{}

\begin{document}

\def\Pr{{\sf P}}
\def\E{{\sf E}}
\def\Var{{\sf Var}}
\def\eps{\varepsilon}
\def\wt{\widetilde}
\def\wh{\widehat}

\def\AA{\mathcal{A}}\def\BB{\mathcal{B}}\def\CC{\mathcal{C}}
\def\DD{\mathcal{D}}\def\EE{\mathcal{E}}\def\FF{\mathcal{F}}
\def\GG{\mathcal{G}}\def\HH{\mathcal{H}}\def\II{\mathcal{I}}
\def\JJ{\mathcal{J}}\def\KK{\mathcal{K}}\def\LL{\mathcal{L}}
\def\MM{\mathcal{M}}\def\NN{\mathcal{N}}\def\OO{\mathcal{O}}
\def\PP{\mathcal{P}}\def\QQ{\mathcal{Q}}\def\RR{\mathcal{R}}
\def\SS{\mathcal{S}}\def\TT{\mathcal{T}}\def\UU{\mathcal{U}}
\def\VV{\mathcal{V}}\def\WW{\mathcal{W}}\def\XX{\mathcal{X}}
\def\YY{\mathcal{Y}}\def\ZZ{\mathcal{Z}}

\def\Ab{\mathbf{A}}\def\Bb{\mathbf{B}}\def\Cb{\mathbf{C}}
\def\Db{\mathbf{D}}\def\Eb{\mathbf{E}}\def\Fb{\mathbf{F}}
\def\Gb{\mathbf{G}}\def\Hb{\mathbf{H}}\def\Ib{\mathbf{I}}
\def\Jb{\mathbf{J}}\def\Kb{\mathbf{K}}\def\Lb{\mathbf{L}}
\def\Mb{\mathbf{M}}\def\Nb{\mathbf{N}}\def\Ob{\mathbf{O}}
\def\Pb{\mathbf{P}}\def\Qb{\mathbf{Q}}\def\Rb{\mathbf{R}}
\def\Sb{\mathbf{S}}\def\Tb{\mathbf{T}}\def\Ub{\mathbf{U}}
\def\Vb{\mathbf{V}}\def\Wb{\mathbf{W}}\def\Xb{\mathbf{X}}
\def\Yb{\mathbf{Y}}\def\Zb{\mathbf{Z}}

\def\ab{\mathbf{a}}\def\bb{\mathbf{b}}\def\cb{\mathbf{c}}
\def\db{\mathbf{d}}\def\eb{\mathbf{e}}\def\fb{\mathbf{f}}
\def\gb{\mathbf{g}}\def\hb{\mathbf{h}}\def\ib{\mathbf{i}}
\def\jb{\mathbf{j}}\def\kb{\mathbf{k}}\def\lb{\mathbf{l}}
\def\mb{\mathbf{m}}\def\nb{\mathbf{n}}\def\ob{\mathbf{o}}
\def\pb{\mathbf{p}}\def\qb{\mathbf{q}}\def\rb{\mathbf{r}}
\def\sb{\mathbf{s}}\def\tb{\mathbf{t}}\def\ub{\mathbf{u}}
\def\vb{\mathbf{v}}\def\wb{\mathbf{w}}\def\xb{\mathbf{x}}
\def\yb{\mathbf{y}}\def\zb{\mathbf{z}}

\def\Abb{\mathbb{A}}\def\Bbb{\mathbb{B}}\def\Cbb{\mathbb{C}}
\def\Dbb{\mathbb{D}}\def\Ebb{\mathbb{E}}\def\Fbb{\mathbb{F}}
\def\Gbb{\mathbb{G}}\def\Hbb{\mathbb{H}}\def\Ibb{\mathbb{I}}
\def\Jbb{\mathbb{J}}\def\Kbb{\mathbb{K}}\def\Lbb{\mathbb{L}}
\def\Mbb{\mathbb{M}}\def\Nbb{\mathbb{N}}\def\Obb{\mathbb{O}}
\def\Pbb{\mathbb{P}}\def\Qbb{\mathbb{Q}}\def\Rbb{\mathbb{R}}
\def\Sbb{\mathbb{S}}\def\Tbb{\mathbb{T}}\def\Ubb{\mathbb{U}}
\def\Vbb{\mathbb{V}}\def\Wbb{\mathbb{W}}\def\Xbb{\mathbb{X}}
\def\Ybb{\mathbb{Y}}\def\Zbb{\mathbb{Z}}

\def\R{{\mathbb R}}
\def\Q{{\mathbb Q}}
\def\Z{{\mathbb Z}}
\def\N{{\mathbb N}}
\def\C{{\mathbb C}}
\def\U{{\mathbb U}}
\def\Ge{{\mathbb G}}
\def\Ha{{\mathbb H}}

\def\lunl{[\hskip-1pt[}
\def\runl{]\hskip-1pt]}
\def\one{{\mathbbm1}}

\def\sbd#1{{#1}^\text{\rm sub}}

\maketitle


\begin{abstract}
A property of finite graphs is called nondeterministically testable
if it has a ``certificate'' such that once the certificate is
specified, its correctness can be verified by random local testing.
In this paper we study certificates that consist of one or more unary
and/or binary relations on the nodes, in the case of dense graphs.
Using the theory of graph limits, we prove that nondeterministically
testable properties are also deterministically testable.
\end{abstract}

\section{Introduction}

Let $\PP$ be a property of finite simple graphs (i.e., a class of
finite simple graphs closed under isomorphism). We say that $\PP$ is
{\it testable}, if there exists another property $\TT$ (called a {\it
test property}) satisfying the following conditions:

\begin{itemize}
\item if a graph $G$ has property $\PP$, then for all $1\le r\le
    |V(G)|$, a random induced subgraph on $r$ nodes (chosen
    uniformly among all such induced subgraphs) has property
    $\TT$ with probability at least $2/3$, and

\item for every $\eps>0$ there is an $r_\eps\ge 1$ such that if
    $G$ is a graph whose edit distance from $\PP$ is at least
    $\eps |V(G)|^2$, then for all $r_\eps\le r\le |V(G)|$, a
    random induced subgraph on $r$ nodes has property $\TT$ with
    probability at most $1/3$.
\end{itemize}

This notion of testability is often called {\it oblivious testing},
which refers to the fact that no information about the size of $G$ is
assumed. It is easy to see that if $\PP$ is a testable property such
that arbitrarily large graphs can have the property, then there must
exist a graph with property $\PP$ with any sufficiently large number
of nodes. The definition extends trivially to graphs whose edges are
oriented, and whose nodes and/or edges are colored with a fixed
finite number $k$ of colors.

Testability of graph properties was introduced by Rubinfeld and Sudan
\cite{RS} and Goldreich, Goldwasser and Ron \cite{GGR}. There are
many graph properties that are known to be testable \cite{GGR,AFKS};
(see e.g.~\cite{Fisch} for a survey and \cite{Gold} for a collection
of more recent surveys). One surprisingly general sufficient
condition was found by Alon and Shapira \cite{ASh}: {\it Every
hereditary graph property is testable.} (A graph property is called
{\it hereditary}, if it is inherited by induced subgraphs.)

Various characterizations of testable properties are known
\cite{AFNS,LSz4}, but they are not simple to state. The goal of this
paper is to prove a characterization, which is useful as a sufficient
condition in a number of cases.

Let $L$ be a directed graph whose nodes and edges are colored. Let us
make some simplifying assumptions that don't change the results. We
assume that every pair of nodes is connected by two, oppositely
directed edges. If some edges are missing, we can add them colored
with an additional color. If some edges have larger (but bounded)
multiplicity, we can use colors to indicate this. Furthermore, we may
get rid of the node colors by coloring every edge $e$ with the triple
$(a,b,c)$, where $a$ is the original color of the edge, $b$ is the
color of its head, and $c$ is the color of its tail. Edge-colorings
obtained this way have a special property (each edge incident with a
node $v$ should carry the same information about the color of $V$),
but this consistency is a testable property. We call a complete
digraph whose edges are colored with $1,\dots,k$ briefly a {\it
$k$-colored digraph}.

Given a $k$-colored digraph $L$ and a positive integer $m \le k$, we
can get an ordinary graph from $L$ by keeping only the edges with
colors $1,\dots,m $, and then forgetting the coloring and the
orientation. We call this graph $L'$ the ``shadow'' of $L$. If $\QQ$
is a property of colored directed graphs, then we define
$\QQ'=\{L':~L\in\QQ\}$.

A graph property $\PP$ is {\it nondeterministically testable}, if
there exist two integers $k\ge m \ge1$ and a property $\QQ$ of
$k$-colored digraphs such that $\QQ$ is testable and $\QQ'=\PP$. In
other words, $G$ has property $\PP$ if and only if we can orient its
edges (in one or both directions), color them with $m$ colors, add
all the missing oriented edges, and color them with further $k-m$
colors, so that the resulting $k$-colored digraph has property $\QQ$.
We call such an orientation and coloring a {\it certificate} for
$\PP$.

Instead of a $k$-coloring, we could specify $k$ binary relations on
$V(G)$ as a certificate (this would be more in the spirit of
mathematical logic). The fact that in a coloring they are disjoint
and partition $V(G)$ and $\binom{V(G)}{2}$, respectively, can be
easily tested. Conversely, such a system of relations defines a
$2^k$-coloring. As long as we are not concerned with efficiency,
these two ways of looking at certificates are equivalent.

Clearly every testable property is nondeterministically testable
(choosing $k=2$, $m=1$). Our main result asserts the converse.

\begin{theorem}\label{THM:MAIN}
A graph property is nondeterministically testable if and only if it
is testable.
\end{theorem}

One could say that this theorem shows that ``P=NP'' for property
testing in dense graphs. The proof uses the theory of graph limits as
developed in \cite{BCLSV1,LSz1}, and its connection with property
testing \cite{LSz4}.

\ignore{ To give a rough sketch of the proof technique, let $\PP$ be
a nondeterministically testable property that is the shadow of a
testable property $\QQ$, but itself is not testable. Then there are
two sequences of graphs $(\wh{G}_1,\wh{G}_2,\dots)$ and
$(H_1,H_2,\dots)$ such that the sampling distance of $G_n$ and $H_n$
tends to $0$ as $n\ti\infty$, and every $G_n$ has property $\PP$, but
every $H_n$ is $\eps$-far from the property in edit distance (for
some $\eps>0$). Each $G_n$ is the shadow of a colored digraph
$L_n\in\QQ$. Using the theory of graph limits, we may assume that
$G_n$ and $H_n$ converge to the same limit object $W$ (a graphon),
and $L_n$ converges to a limit object $U$ (a colored digraphon) whose
shadow if the graphon $W$. The main step in the proof is to pull back
the orientation and coloring of $U$ onto the graphs $H_n$, which will
represent them as shadows of}

\section{Preliminaries}

\subsection{Convergence and limits}

Convergence of a sequence of dense finite graphs was defined by
Borgs, Chayes, Lov\'asz, S\'os and Vesztergombi \cite{BCLSV0,BCLSV1}.
Graphons were introduced by Lov\'asz an Szegedy in \cite{LSz1} as
limits of convergent sequences of finite graphs. We have to extend
these notions to colored digraphs; this was essentially done in
\cite{LSz9}, but we use here a little different (simpler)
terminology.

For a graph $G$ and integer $r\le |V(G)|$, let us select an ordered
$r$-tuple of nodes of $G$ randomly and uniformly (without
repetition). Let $\Gbb(r,G)$ denote the subgraph induced by these
nodes. If $G$ is colored and directed, then $\Gbb(r,G)$ is also a
colored digraph in the obvious way.

We say that a sequence of $k$-colored digraphs $L_n$ is {\it
convergent}, if $|V(G_n)|\to\infty$, and for every $r\ge1$, the
distribution of $\Gbb(r,L_n)$ tends to a limit as $n\to\infty$. Note
that this distribution is over a finite set, so it does not matter in
which norm we want it to converge.

The limit object of a convergent sequence of simple graphs can be
described as a symmetric measurable function $W:~[0,1]^2\to[0,1]$,
called a {\it graphon}. We will need the more general notion of a
{\it kernel}, a bounded symmetric measurable function
$W:~[0,1]^2\to\R$. Dropping the condition of symmetry, we get {\it
digraphons} and {\it dikernels}.

For a sequence of $k$-colored digraphs, the limit object is a bit
more complicated: it consists of $k$ digraphons $(W^1,\dots,W^k)$
such that $\sum_h W^h=1$. We call the $k$-tuple
$\mathbf{W}=(W^1,\dots,W^k)$ a {\it $k$-digraphon}. We define a {\it
$k$-dikernel} analogously.

Let $L$ be a $k$-colored digraph with $V(L)=[n]$. Let ${\SS_n}$
denote the partition of $[0,1]$ into $n$ intervals $S_1,\dots,S_n$ of
equal length. We can associate with $L$ a $k$-digraphon
$\mathbf{W}_L= (W_L^1,\dots,W_L^k)$, where
\[
\mathbf{W}_L^h(x,y)=
  \begin{cases}
    1, & \text{if $x\in S_i$, $y\in S_j$, and the color of $ij$ is $h$}, \\
    0, & \text{otherwise}.
  \end{cases}
\]
More generally, we can consider a {\it fractionally $k$-colored
digraph} $H$ in which we have $k$ nonnegative weights
$\beta^1(i,j),\dots, \beta^k(i,j)$ for every ordered pair $(i,j)$ of
nodes, with $\sum_h \beta^h(i,j)=1$. We consider every $k$-colored
digraph as a special case, where $\beta^h$ is the indicator function
of edge color $h$. For a fractionally $k$-colored digraph $H$, we
define the $k$-digraphon $\Wb_H$ in the obvious way.

We can sample a $k$-colored digraph $\Gbb(r,\Wb)$ on node set $[r]$
from a $k$-digraphon $\mathbf{W}$ as follows: we choose $r$
independent random points $X_1,\dots,X_r\in[0,1]$ uniformly, and we
color a pair $(i,j)$ with color $h$ with probability $W^h(X_i,X_j)$
(independently for different pairs of nodes).

The following fact is proved in \cite{LSz9}.

\begin{prop}\label{PROP:CONV-COLORED}
Let $L_n$ be a convergent sequence of $k$-colored digraphs. Then
there is a $k$-digraphon $\mathbf{W}$ such that
$\Gbb(r,L_n)\to\Gbb(r,\mathbf{W})$ in distribution.
\end{prop}

\noindent We write $L_n\to \mathbf{W}$ if this holds.

\medskip

Let $W$ be a dikernel, and let $\JJ=\{S_1,\dots,S_m\}$ be a partition
of $[0,1]$ into measurable sets with positive measure. We denote by
$W_\JJ$ the dikernel obtained by averaging $W$ in every rectangle
$S_i\times S_j$. More precisely, for $x\in S_i$ and $y\in S_j$ we
define
\[
W_\JJ(x,y)= \frac{1}{\lambda(S_i)\lambda(S_j)}\int\limits_{S_i\times
S_j} W(u,z)\,du\,dz.
\]

We quote a well-known fact:

\begin{prop}\label{PROP:AS-APPROX}
For every dikernel $W$, we have $W_{\SS_n}\to W$ $(n\to\infty)$
almost everywhere.
\end{prop}

\subsection{Many distances}

\noindent{\bf Cut norm and cut distance.} Convergence to a
$k$-digraphon can be described in more explicit forms. Let us start
with recalling the {\it cut-norm-distance} of two graphs $G$ and $G'$
on the same node set $V$ (introduced by Frieze and Kannan \cite{FK}):
\[
d_\square(G,G')= \max_{S,T\subseteq V}
\frac{|e_G(S,T)-e_{G'}(S,T)|}{|V|^2},
\]
where $e_G(S,T)$ denotes the number of edges with one endpoint in $S$
and the other in $T$. A related notion for dikernels is the {\it cut
norm}
\[
\|W\|_\square = \sup_{S,T\subseteq[0,1]}\biggl|\int\limits_{S\times
T} W(x,y)\,dx\,dy\biggr|.
\]
The cut norm defines a distance function between two dikernels in the
usual way by $d_\square(U,W)=\|U-W\|_\square$.

The cut norm has many nice properties (see \cite{BCLSV1}), of which
we need the following (Lemma 2.2 in \cite{LSz4}):

\begin{prop}\label{PROP:PROD-CONT}
Let $W_n$ $(n=1,2,\dots)$ be a sequence of uniformly bounded
dikernels such that $\|W_n\|_\square\to0$. Then for every bounded
measurable function $Z:~[0,1]^2\to\R$, we have
$\|W_nZ\|_\square\to0$.
\end{prop}

\begin{proof}
If $Z$ is the indicator function of a rectangle, the conclusion
follows from the definition of the $\|.\|_\square$ norm. Hence the
conclusion follows for stepfunctions, since they are linear
combinations of a finite number of indicator functions of rectangles.
Then it follows for all integrable functions, since they are
approximable in $L_1([0,1]^2)$ by stepfunctions.
\end{proof}

From the point of view of graph limits, however, a kernel is only
relevant up to a measure preserving transformation of $[0,1]$. Hence
it is often more natural to consider the following distance notion,
which we call the {\it cut distance}:
\[
\delta_\square(U,W) = \inf_{\phi,\psi}\|U^\phi-W^\psi\|_\square,
\]
where $\phi,\psi:[0,1]\to[0,1]$ are measure preserving maps, and
$U^\phi(x,y)=U(\phi(x),\phi(y))$. This defines a pseudometric on the
set of kernels (it is only a pseudometric, since different kernels
may have distance $0$). An important fact is that endowing the space
of all graphons with this pseudometric makes it compact \cite{LSz3}.

For two graphs $G$ and $G'$ (not necessarily with the same number of
nodes) we define
\[
\delta_\square(G,G')=\delta_\square(W_G,W_{G'}).
\]
(There is a finite description of this in terms of the optimum of a
quadratic program, but it is quite complicated, and for us this less
explicit definition will be sufficient.) It is easy to see that if
$V(G)=V(G')$, then
\[
\delta_\square(G,G')\le d_\square(G,G').
\]

These distance notions can be extended to colored graphs and kernels.
For two fractionally $k$-colored digraphs $H$ and $H'$ on the same
node set $V$, with edgeweights $\beta_H^h(i,j)$ and
$\beta_{H'}^h(i,j)$ ($h=1,\dots,k$), let
\[
d_\square(H,H')= \frac1{|V|^2}\sum_{h=1}^k \max_{S,T\subseteq V}
\biggl|\sum_{i\in S\atop j\in T} (\beta_H^h(i,j)-
\beta_{H'}^h(i,j))\biggr|.
\]

We generalize the cut-norm-distance to two $k$-digraphons
$\Ub=(U^1,\dots,U^k)$ and $\Wb=(W^1,\dots,W^k)$:
\begin{equation}\label{EQ:D-CUT}
d_\square(\Ub,\Wb) = \sum_{h=1}^k \|U^h-W^h\|_\square.
\end{equation}
Similarly as above, we need the ``unlabeled'' cut distance
\begin{equation}\label{EQ:DELTA-CUT}
\delta_\square(\Ub,\Wb) = \inf_{\phi,\psi} d_\square(\Ub^\phi,\Wb^\psi)
\end{equation}
(where $\Ub^\psi$ is obtained from $\Ub$ by substituting $\phi(x)$
for $x$ in each of the $2k$ functions constituting $\Ub$).

For two fractionally $k$-colored digraphs $H$ and $H'$ (not
necessarily with the same number of nodes) we define
\[
\delta_\square(H,H')=\delta_\square(\Wb_H,\mathbf{W}_{H'}).
\]

The following result is proved (in a more general form) in
\cite{LSz9}.

\begin{prop}\label{PROP:CONV-COLORED-CUT}
Let $L_n$ be a sequence of $k$-colored digraphs, and let $\Wb$ be a
$k$-digraphon. Then $L_n\to\Wb$ if and only if
$\delta_\square(\Wb_{L_n},\Wb)\to0$.
\end{prop}

We cannot claim convergence in the $d_\square$ distance, since
$\Wb_{L_n}$ depends on the labeling of the nodes of $L_n$, while the
convergence $L_n\to\Wb$ does not. However, at least in the case of
simple graphs, the following stronger version is true (\cite{BCLSV1},
Theorem 4.16):

\begin{prop}\label{PROP:CONV-NORM}
Let $G_n$ be a sequence of graphs, and let $U$ be a graphon such that
$G_n\to U$. Then the graphs $G_n$ can be labeled so that
$\|W_{G_n}-U\|_\square\to0$.
\end{prop}

\noindent{\bf Edit distance.} From the point of view of property
testing, the ``edit distance'' is very important (in fact, from the
point of view of analysis, property testing is about the interplay
between the edit distance and the cut distance). For two graphs $G$
and $G'$ on the same set of nodes $V=V(G)=V(G')$, their {\it edit
distance} is defined by
\[
d_1(G,G') = \frac{|E(G)\triangle E(G')|}{|V|^2}.
\]
We generalize this for two $k$-colored digraphs $L$ and $L'$ on the
same node set:
\[
d_1(L,L') =\frac{D_2}{|V|^2},
\]
where $D_2$ is the number of edges colored differently, in $L$ and
$L'$.

For two kernels, their edit distance is just their $L_1$-distance as
functions. For two $k$-digraphons, their edit distance is defined by
a formula very similar to \eqref{EQ:D-CUT}:
\[
d_1(\Ub,\Wb) = \sum_{h=1}^k \|U^h-W^h\|_1.
\]
Similarly to \eqref{EQ:DELTA-CUT}, we could define the unlabeled
version of the edit distance, but we don't need it in this paper.

The following (easy) characterization of testability of graph
properties was formulated in \cite{LSz4}, Theorem 3.20.

\begin{prop}\label{PROP:TEST-CHAR}
A graph property $\PP$ is testable if and only if for any sequence
$(G_n)$ of graphs with $|V(G_n)|\to\infty$, the condition
$\delta_\square(G_n,\PP)\to 0$ implies that $d_1(G_n,\PP)\to 0$.
\end{prop}

We note that the condition says that $G_n$ is close to some graph in
$\PP$ (not necessarily with the same number of nodes) in the
$\delta_\square$ distance, while the conclusion is that it must be
close to a graph in $\PP$ on the same node set in the edit distance.

\section{Main proof}

We start with a randomized construction to obtain a $k$-colored
digraph from a fractionally $k$-colored digraph $H$: we color every
edge $ij\in \binom{[n]}{2}$ with color $h$ with probability
$\beta^h(i,j)$. For different pairs $i,j$ we make an independent
decision. We denote this random $k$-colored digraph by $\Lbb(H)$.

\begin{lemma}\label{LEM:LH-CLOSE}
Let $H$ be a fractionally $k$-colored digraph on $n$ nodes. Then
\[
d_\square(H,\Lbb(H))\leq {\frac{10k}{\sqrt n}}
\]
with probability at least $1-ke^{-n}$.
\end{lemma}

\begin{proof}
For two colors, this is just Lemma 4.3 in \cite{BCLSV1}. For general
$k$, it follows by applying this fact to each color separately.
\end{proof}

The main step in the proof of Theorem \ref{THM:MAIN} is the following
lemma.

\begin{lemma}\label{LEM:PULLBACK}
Let $\Wb=(W^1,\dots,W^k)$ be a $k$-digraphon, and suppose that
$U=\sum_{h=1}^m W^h$ is symmetric (where $1\le m\le k)$. Let $F_n$ be
a sequence of simple graphs such that $F_n\to U$. Then there exist
$k$-colored digraphs $J_n$ on $V(F_n)$ such that $J_n'=F_n$ and
$J_n\to\Wb$.
\end{lemma}

\begin{proof}
First, we construct a fractionally $k$-colored digraph $H_n$ on
$V(F_n)$. To keep the notation simple, assume that $V(F_n)=[n]$. By
Proposition \ref{PROP:CONV-NORM}, we can choose the labeling of the
nodes of each $F_n$ so that $\|W_{F_n}-U\|_\square\to 0$.

For every pair $i,j\in [n]$, we define
\[
\beta^h(i,j) =
  \begin{cases}
    \displaystyle (A_n)_{ij} \frac{W^h_{\SS_n}(x,y)}{U_{\SS_n}(x,y)}
    & \text{if $1\le h\le m $}, \\
    \displaystyle (1-(A_n)_{ij})\frac{W^h_{\SS_n}(x,y)}{1-U_{\SS_n}(x,y)}
    & \text{if $m+1\le h\le k$},
  \end{cases}
\]
where $A_n$ denotes the adjacency matrix of $G_n$ and $x\in S_i$ and
$y\in S_j$ (these numbers are independent of the choice of $x$ and
$y$). It is easy to check that $\sum_h \beta^h(i,j)=1$ for all
$i\not=j$. We show that for the fractionally $k$-colored digraph
$H_n$ constructed this way, we have
\begin{equation}\label{EQ:WNW}
d_\square(\Wb_{H_n},\Wb)\to 0 \qquad(n\to\infty).
\end{equation}
Here we have
\[
d_\square(\Wb_{H_n},\Wb) = \sum_{h=0}^k \|W^h_{H_n}-W^h\|_\square,
\]
and so it suffices to prove that $ \|W^h_{H_n}-W^h\|_\square\to0$ for
every fixed $h$. We describe the proof for $h\le m $; the other case
is analogous. Since $0\le W^h\le U$, we can write $W^h=UZ$, where
$0\le Z\le 1$, and $Z=0$ if $U=0$. Then we have
\[
\|W^h_{H_n}-W^h\|_\square = \sup_{S,T\subseteq [0,1]}
\Bigl|\int\limits_{S\times T} W^h_{H_n}-W^h\Bigr|.
\]
Substituting from the definition,
\[
\int\limits_{S\times T} (W^h_{H_n}-W^h) = \int\limits_{S\times T}
\Bigl(W_{F_n}(x,y) \frac{W^h_{\SS_n}(x,y)}{U_{\SS_n}(x,y)}
-W^h(x,y)\Bigr)\,dx\,dy.
\]
We split this integral as follows:
\begin{align}\label{EQ:NORM-1}
\int\limits_{S\times T\atop U=0}
W_{F_n}\frac{(UZ)_{\SS_n}}{U_{\SS_n}}+ \int\limits_{S\times T\atop
U\not=0} W_{F_n}\Bigl(\frac{(UZ)_{\SS_n}}{U_{\SS_n}}-Z\Bigr)+
\int\limits_{S\times T} (W_{F_n}-U) Z.
\end{align}
The first term, which is nonnegative, can be estimated as follows:
\begin{align*}
\int\limits_{S\times T\atop U=0}
W_{F_n}\frac{(UZ)_{\SS_n}}{U_{\SS_n}} &\le \int\limits_{S\times T}
W_{F_n}\one_{U=0} = \int\limits_{S\times T} (W_{F_n}-U)\one_{U=0} \le
\|(W_{F_n}-U)\one_{U=0}\|_\square.
\end{align*}
Here the right hand side tends to $0$ by Proposition
\ref{PROP:PROD-CONT}. The second term can be estimated like this. By
Proposition \ref{PROP:AS-APPROX}, we have $(UZ)_{\SS_n}\to UZ$ and
$U_{\SS_n}\to U$ almost everywhere. Hence $(UZ)_{\SS_n}/U_{\SS_n}\to
Z$ in almost every point where $U\not=0$. Since the integrand is
bounded, this implies that
\[
\Bigl|\int\limits_{S\times T\atop U\not=0}
W_{F_n}\Bigl(\frac{(UZ)_{\SS_n}}{U_{\SS_n}}-Z\Bigr)\Bigr| \le
\int\limits_{U\not=0}
W_{F_n}\Bigl|\frac{(UZ)_{\SS_n}}{U_{\SS_n}}-Z\Bigr|\to 0.
\]
Finally for the third term in \eqref{EQ:NORM-1}, we have
\[
\Bigl|\int\limits_{S\times T} (W_{F_n}-U) Z\Bigr|\le
\|(W_{F_n}-U)Z\|_\square,
\]
and here the right hand side tends to $0$, again by Proposition
\ref{PROP:PROD-CONT}. This proves \eqref{EQ:WNW}.

To complete the proof of the lemma, we consider the $k$-colored
digraphs $J_n=\Lbb(H_n)$. By Lemma \ref{LEM:LH-CLOSE}, we have
\begin{equation}\label{EQ:1}
d_\square(J_n,H_n)\le \frac{10k}{\sqrt{n}}
\end{equation}
with probability at least $1-ke^{-n}$. Since $\sum_ne^{-n}$ is
convergent, the Borel--Cantelli Lemma implies that almost surely
\eqref{EQ:1} holds for all but a finite number of indices $n$.
Choosing the $J_n$ so that this occurs, we have
$d_\square(J_n,H_n)=d_\square(\Wb_{J_n},\Wb_{H_n})\to0$, and hence
$d_\square(\Wb_{J_n},\Wb)\to 0$.
\end{proof}

\begin{proof*}{Theorem \ref{THM:MAIN}}
Let $\PP$ be a nondeterministically testable property; we show that
it is testable. By Proposition \ref{PROP:TEST-CHAR} it suffices to
prove that if $(F_n)$ is a sequence of graphs such that
$d_\square(F_n,\PP)\to0$, then $d_1(F_n,\PP)\to0$.

Since $\PP$ is nondeterministically testable, there are integers
$1\le m \le k$ and a testable property $\QQ$ of $k$-colored digraphs
such that $\PP=\QQ'$. Let $G_n\in\PP$ such that
$d_\square(F_n,G_n)\to0$. Since $G_n\in\PP$, there are $k$-colored
digraphs $L_n\in\QQ$ such that $G_n=L_n'$.

We may assume that the union of colors $1,\dots,m$ contains every
edge of $G_n$ in both directions; else, we refine the coloring so
that no edge in $G$ and in its complement gets the same color in any
direction (this doubles the number of colors at most). By selecting a
subsequence, we may assume that the sequence $(L_n)$ is convergent.
Let $\Wb$ be a $k$-digraphon representing its limit, and let
$U=\sum_{h=1}^m W^h$. Then $G_n\to U$. From
$\delta_\square(G_n,F_n)\to0$ we see that $F_n\to U$.

Now we invoke Lemma \ref{LEM:PULLBACK}, and construct $k$-colored
digraphs $J_n$ such that $J_n'=F_n$ and $J_n\to \Wb$. Hence
$d_\square(J_n,\QQ)\to0$. Since $\QQ$ is testable, this implies that
$d_1(J_n,\QQ)\to0$, and so we can change the color of $o(n^2)$ edges
in $J_n$ so that the resulting $k$-colored digraph $M_n$ belongs to
$\QQ$. But then $M_n'\in\PP$, and $M_n'$ differs from $F_n$ in
$o(n^2)$ edges only, so $d_1(F_n,\PP)\le d_1(F_n,M_n')\to0$.
\end{proof*}

\section{Applications}

There are many graph properties that can be certified by a
node-coloring: $k$-colorable graphs, split graphs, etc. Many of these
properties are hereditary, and so their testability follows also by
the Alon--Shapira Theorem mentioned in the introduction. Here we
formulate some consequences for non-hereditary graph properties.

One of the first nontrivial results about property testing concerned
the maximum cut. Let us derive one version. The property of a
$2$-node-colored graph $G$ that ``at least $c|V(G)|^2$ edges connect
nodes with different colors'' is trivially testable, and hence:

\begin{corollary}\label{COR:MAXCUT}
Let $0<c<1$. The property of a graph $G$ that its maximum cut
contains at least $c|V(G)|^2$ edges is testable. Similarly for
maximum bisection.
\end{corollary}

In their paper \cite{GGR}, Theorem 9.1, Goldreich, Goldwasser and Ron
prove the testability of more general properties, namely the
existence of multiway cuts with upper and lower bounds on the sizes
of partition classes as well as on edge densities between parts. The
existence of such a cut can be certified by a node-coloring, and so
this property is trivially nondeterministically testable. So their
general result (without explicit bounds on the sample size) follows
from Theorem \ref{THM:MAIN}.

Alon, Fischer, Krivelevich and Szegedy \cite{AFKS} prove that a graph
property is testable, provided it is expressible in the form $\exists
x_1\dots\exists x_a\forall y_1\dots\forall y_b
\Phi(x_1,\dots,x_a,y_1,\dots,y_b)$, where the $x_i$ and $y_j$ are
variables ranging over nodes, and $\Phi$ is a (quantifier-free)
Boolean expression involving equality and adjacency of the variables
$x_i$ and $y_j$. They also give an example showing that graph
properties defined by more general first order sentences (with more
quantifier alternations) are not necessarily testable.

To relate this result to ours, let us start with noticing that
properties expressible by universal first-order formulas $\forall
y_1\dots\forall y_b \Phi(y_1,\dots,y_b)$ are exactly those
expressible by a finite number of excluded induced subgraphs. (Such
properties are hereditary, and this shows that the Alon--Shapira
Theorem described in the Introduction is a generalization of this
special case in another direction.) In the general case, roughly
speaking, they eliminate the existential quantifiers by encoding them
into a node-coloring: the color of a node expresses to which of the
nodes $x_i$ it is connected, along with the subgraph induced by the
nodes $x_i$. Hence such a property is nondeterministically testable:
the certificate is this coloring.

Our result implies a more general testability condition in terms of
logical formulas:

\begin{corollary}\label{COR:SECOND}
Let $\PP$ be a graph property expressible by a second-order formula
of the form $\exists S_1\dots\exists S_c \exists x_1\dots\exists
x_a\forall y_1\dots\forall y_b
\Phi(S_1,\dots,S_c,x_1,\dots,x_a,y_1,\dots,y_b)$, where the $S_i$ are
variables ranging over unary or binary relations, $x_i$ and $y_j$ are
variables ranging over nodes, and $\Phi$ is a (quantifier-free)
Boolean expression involving equality, adjacency, and the relations
$S_i$ of the variables $x_i$ and $y_j$. Then $\PP$ is testable.
\end{corollary}

In \cite{LSz4}, the {\it upward closure} of a graph property $\PP$
was defined as the graph property $\PP^\uparrow$ consisting of those
graphs that have a spanning subgraph in $\PP$. Suppose that $\PP$ is
testable, then the property of $3$-edge-colored graphs that ``edges
with color $1$ form a graph with property $\PP$'' is testable, and
hence:

\begin{corollary}\label{COR:UPWARD}
The upward closure of a testable graph property is testable.
\end{corollary}

Suppose again that $\PP$ is testable, then the property of
$4$-edge-colored graphs that ``edges with colors $1$ and $2$ form a
graph with property $\PP$, and edges with colors $2$ and $3$ are
fewer than $|V(G)|^2/100$'' is testable, and hence:

\begin{corollary}\label{COR:CLOSE}
If $\PP$ is a testable property, then the property that ``we can
change at most $|V(G)|^2/100$'' edges to get a graph with property
$\PP$'' is also testable.
\end{corollary}

\section{Parameter estimation}

Let $f$ be a bounded graph parameter (i.e., a function defined on
simple graphs, invariant under isomorphism). We say that $f$ is {\it
estimable}, if for every $\eps,\delta>0$ there is a positive integer
$k$ such that if $G$ is a graph with at least $k$ nodes and we select
a random $k$-set $X\subseteq V(G)$, then
\begin{equation}\label{EQ:PARAM-TEST}
\Pr(|f(G)-f(G[X])|>\eps)<\delta.
\end{equation}

We define an estimable parameter of edge-colored digraphs similarly.
If $g$ is such a parameter, then we can define
\[
g'(G)=\max\{g(L):~L'=G\}.
\]

An argument very similar to the proof of Theorem \ref{THM:MAIN} above
gives:

\begin{theorem}\label{THM:MAX}
If $g$ is an estimable parameter of $k$-colored digraphs, then $g'$
is estimable as well.
\end{theorem}

We could of course replace the maximum in the definition of $g'$ by
minimum.

As an example, let us consider graphs $L$ whose nodes are $2$-colored
red and blue, and let $g(L)$ denote the number of $2$-colored edges,
divided by $|V(G)|^2$. Then $g$ is trivially estimable. The simple
graph parameter $g'$ is the maximum cut (normalized), so this is
estimable.

\section{Concluding remarks}

There are several possible analogues and extensions of our results.
One could consider certificates in the form of $t$-ary relations for
any $t$. One could then allow hypergraphs instead of the original
graphs. A limit theory for hypergraphs is available (Elek and Szegedy
\cite{ESz}, and we expect our main result to generalize to
hypergraphs; however, several of the auxiliary results we have made
use of have not been extended, and a full proof will take further
research.

A generalization in a different direction would be to consider,
instead of coloring, node and edge decorations from a compact
topological space. For example, the property of being a threshold
graph can be certified by a decoration of the nodes by numbers from
$[0,1]$. The limit theory for graphs has been extended to compact
decorations \cite{LSz9}; perhaps our main result extends too, but
this takes further research.

We should point out that our results are non-effective, they don't
provide any explicit bound on how large a sample size must be chosen
for a given error bound. In this sense, what we have given is a pure
existence proof of an algorithm. From a practical point of view, this
does not make much difference from related results based on the
Regularity Lemma, but from a theoretical point of view, it would be
interesting to determine whether Theorem \ref{THM:MAIN} can be proved
with an effective bound.

Finally, let us mention that the situation is quite different in the
case of graphs with bounded degree (for which a limit theory
analogous to the dense case is available, and property testing has
been extensively studied). Here the sampling method is to select $r$
random nodes (uniformly), and explore their neighborhoods of depth
$r$. The property of a graph $G$ that ``$G$ is the disjoint union of
two graphs on at least $|V(G)|/3$ nodes'' can be certified by
coloring the nodes in these two graphs with different colors, so this
property is nondeterministically testable. On the other hand,
sampling will not distinguish between an expander graph and the
disjoint union of two copies of it, so this property is not testable.

\bigskip

\noindent{\bf Acknowledgements.} Research was supported by the
European Research Council Grant No.~227701 and by the National
Science Foundation under agreement No. DMS-0835373. Hospitality of
the Institute for Advanced Study is gratefully acknowledged. Any
opinions and conclusions expressed in this material are those of the
authors and do not necessarily reflect the views of the NSF or of the
ERC.

\end{document}